\documentclass{amsart}

\input xy

\xyoption{all}

\usepackage{amssymb,amsbsy,amsthm,amsmath,graphicx,epsfig}

\usepackage{mathabx,times}

\newtheorem{thm}{Theorem}[section]
\newtheorem{lem}[thm]{Lemma}
\newtheorem{prop}[thm]{Proposition}

\newtheorem*{thm*}{Theorem}

\theoremstyle{definition}

\newtheorem*{rmks}{Remarks}
\newtheorem{ques}[thm]{Question}

\renewcommand{\rm}[1]{\mathrm{#1}}
\newcommand{\cal}[1]{\mathcal{#1}}

\newcommand{\bbR}{\mathbb{R}}
\newcommand{\bbT}{\mathbb{T}}
\newcommand{\bbZ}{\mathbb{Z}}

\newcommand{\sfE}{E}

\renewcommand{\d}{\mathrm{d}}
\newcommand{\rme}{\mathrm{e}}

\renewcommand{\a}{\alpha}

\newcommand{\eps}{\varepsilon}
\newcommand{\g}{\gamma}

\newcommand{\s}{\sigma}

\newcommand{\ol}[1]{\overline{#1}}

\renewcommand{\to}{\longrightarrow}
\newcommand{\fin}{\nolinebreak\hspace{\stretch{1}}$\lhd$}

\newcommand{\vpartial}{\vec{\partial}}

\begin{document}

\title[Failure of concentration]{On the failure of concentration for the $\ell_\infty$-ball}
\author[Tim Austin]{Tim Austin}
\address{Courant Institute, New York University\\
251 Mercer St, New York, NY 10012\\
U.S.A.}
\thanks{Research supported by a fellowship from the Clay Mathematics Institute}
\email{tim@cims.nyu.edu}
\date{}

\begin{abstract}
Let $(X,d)$ be a compact metric space and $\mu$ a Borel probability on $X$.  For each $N\geq 1$ let $d^N_\infty$ be the $\ell_\infty$-product on $X^N$ of copies of $d$, and consider $1$-Lipschitz functions $X^N\to\bbR$ for $d^N_\infty$.

If the support of $\mu$ is connected and locally connected, then all such functions are close in probability to juntas: that is, functions that depend on only a few coordinates of $X^N$.  This describes the failure of measure concentration for these product spaces, and can be seen as a Lipschitz-function counterpart of the celebrated result of Friedgut that Boolean functions with small influences are close to juntas.
\end{abstract}

\maketitle

\thispagestyle{plain}

\pagestyle{plain}





\section{Introduction}

In this paper, a \textbf{metric probability space} will be a triple $(X,d,\mu)$ in which $d$ is a compact metric on $X$ and $\mu$ is a probability measure on the Borel $\s$-algebra of $(X,d)$.

Given a metric space $(X,d)$ and an integer $N \geq 1$, we shall write $d^N_\infty$ for the $\ell_\infty$-product of $N$ copies of $d$: that is,
\[d^N_\infty((x_n)_{n\leq N},(x'_n)_{n\leq N}) := \max_{n\leq N}\,d(x_n,x_n').\]
If $\mu$ is a probability measure on $X$, then $\mu^N$ denotes its $N$-fold product on $X^N$.

Recall that a sequence of metric probability spaces $(X_N,d_N,\mu_N)$ exhibits \textbf{concentration of measure} if
\begin{eqnarray}\label{eq:conc}
\sup_{\scriptsize{\begin{array}{c}f:X_N\to \bbR\\ \rm{Lip}(f) \leq 1\end{array}}} \mu_N\{|f - \hbox{$\int$} f\,\d\mu_N| \geq \delta\} \to 0 \quad\quad \forall \delta > 0
\end{eqnarray}
as $N\to\infty$ (see, for instance,~\cite{Led--book} and the many references there).

Unless $\mu$ is a Dirac mass, the sequence of metric probability spaces $(X^N,d^N_\infty,\mu^N)$ does not exhibit concentration of measure.  Each coordinate projection $\pi_n:X^N\to X$ is $1$-Lipschitz and pushes $\mu^N$ to $\mu$.  Composing one of these with any non-constant $1$-Lipschitz function $f:X\to\bbR$ witnesses that~(\ref{eq:conc}) is violated, because for small enough $\delta > 0$ the supremum on the left is bounded below by the positive constant $\mu\{|f - \int f \,\d\mu| \geq \delta\}$, independently of $N$.

In light of this, one can ask for a rough description of all functions $X^N \to \bbR$ that have a fixed Lipschitz constant but are not highly concentrated for the measure $\mu^N$.  If the support of $\mu$, $\rm{spt}\,\mu$, is disconnected, then there are many such functions, and no special description is conceivable: for instance, if $X = \{0,1\}$, $d(0,1) = 1$ and $\mu = \frac{1}{2}(\delta_0 + \delta_1)$, then \emph{all} functions $X^N\to [0,1]$ are $1$-Lipschitz for the metric $d^N_\infty$.  This observation may easily be extended to other examples with disconnected $\rm{spt}\,\mu$.

However, the situation changes in case $\rm{spt}\,\mu$ has some connectedness.  This will be made precise in Theorem~\ref{thm:bigmain} below, which will involve the following notions.

First, $(X,d,\mu)$ is (\textbf{locally}) \textbf{connected} if $\rm{spt}\,\mu$ is (locally) connected.

Next, when a probability space $(X,\mu)$ is understood and $n \in [N]$, let $\sfE_{[N]\setminus \{n\}}$ denote the conditional expectation operator on $L_1(\mu^N)$ defined by integrating out the $n^{\rm{th}}$ coordinate:
\[\sfE_{[N]\setminus \{n\}}f(x_1,\ldots,x_N) := \int_X f(x_1,\ldots,x_{n-1},y,x_{n+1},\ldots,x_N)\,\mu(\d y).\]
More generally, for $S = [N]\setminus \{n_1,\ldots,n_k\} \subseteq [N]$ let
\[\sfE_S := \sfE_{[N]\setminus \{n_1\}}\circ \cdots \circ \sfE_{[N]\setminus \{n_k\}},\]
where clearly the order of composition is unimportant.  In this notation, $S$ records those coordinates in $[N]$ on which the dependence of $f$ is retained.

Lastly, for any metric space $(X,d)$ and function $f:X\to\bbR$, the \textbf{non-decreasing modulus of continuity} of $f$ is the function $\omega:[0,\infty)\to[0,\infty]$ defined by
\[\omega(r) := \sup\{|f(x) - f(y)|\,|\ d(x,y) \leq r\}.\]
Clearly $\omega(0) = 0$, and uniform continuity of $f$ is equivalent to continuity of $\omega$ at $0$.

\begin{thm}\label{thm:bigmain}
Let $(X,d,\mu)$ be a connected and locally connected metric probability space, and let $\omega:[0,\infty)\to [0,\infty]$ be non-decreasing.  For every $\eps > 0$ there is some integer $p \geq 1$, depending on $X$, $\eps$ and $\omega$, with the following property. For every $N \geq 1$, if $f:X^N\to \bbR$ has modulus of continuity at most $\omega$ for the metric $d^N_\infty$, then there is $S \subseteq [N]$ with $|S|\leq p$ such that
\[\|f - \sfE_Sf\|_{L_1(\mu^N)} < \eps.\]
\end{thm}

I suspect that this theorem needs only the assumption of connectedness, not local connectedness.

Approximation in $L_1(\mu^N)$ is convenient for the proof, but it implies an approximation in $L_p(\mu^N)$ for any $p < \infty$: this is because $\rm{diam}(X^N,d^N_\infty) = \rm{diam}(X,d) < \infty$ for all $N$, and so the modulus-of-continuity bound implies also a uniform bound on $|f - \int f|$.  However, Theorem~\ref{thm:bigmain} cannot be tightened to an approximation in $L_\infty(\mu^N)$: for instance, on $[0,1]^N$ with Lebesgue measure, the function
\[f(x_1,\ldots,x_N) := \max_{n\leq N} x_n\]
is $1$-Lipschitz, takes values in $(1-\eps,1]$ with probability $1 - (1-\eps)^N$, but is not uniformly close to any function depending on only $N-1$ coordinates.

Let $\bbT:= \bbR/\bbZ$, endowed with the metric inherited from the usual distance on $\bbR$, which we still denote by $|\cdot|$.  Most of the work in proving Theorem~\ref{thm:bigmain} will go towards the following special case:

\begin{thm}\label{thm:smallmain}
For every $\eps > 0$ there is a $q \geq 1$ such that if $f:\bbT^N\to \bbR$ is $1$-Lipschitz for the metric $|\cdot|^N_\infty$, then there is some $S\subseteq [N]$ with $|S|\leq q$ such that
\[\|f - \sfE_Sf\|_{L_1(\bbT^N)} < \eps.\]
\end{thm}

An easy exercise shows that with the above metrics on product spaces, any conditional expectation of the form $\sfE_Sf$ has modulus of continuity not greater than that of $f$ itself.  Therefore both of these theorems give approximations in measure by juntas that are `as continuous' as the original function $f$.

In fact, the method below will give a stronger variant of Theorem~\ref{thm:smallmain}, formulated in terms of a kind of Sobolev norm for $f$: see Theorem~\ref{thm:medmain}.  This stronger version also has a consequence for maps between Hamming metrics, rather than $\ell_\infty$-metrics.  Recall that given $(X,d)$, the associated \textbf{Hamming metric} $d^N_1$ on $X^N$ is defined by
\[d^N_1((x_n)_{n\leq N},(x'_n)_{n\leq N}) := \frac{1}{N}\sum_{n=1}^Nd(x_n,x'_n).\]
Clearly $d^N_1 \leq d^N_\infty$.  The additional consequence of Theorem~\ref{thm:medmain} is as follows.

\begin{thm}\label{thm:biLiprig}
Let $N,M\geq 1$ with $M/N =: \a$.  Then for every $\eps,L > 0$ there is a $q\geq 1$, depending only on $\eps$, $\a$ and $L$, with the following property.  If $F:[0,1]^N\to [0,1]^M$ is $L$-Lipschitz for the Hamming metrics on domain and target, then  there is another $L$-Lipschitz function $G = (G_1,\ldots,G_M):[0,1]^N\to [0,1]^M$ such that
\[\int_{[0,1]^N}|F(x) - G(x)|^M_1\,\d x < \eps,\]
and such that each $G_m$ depends on only $q$ coordinates of $[0,1]^N$.
\end{thm}

The proof of Theorem~\ref{thm:smallmain} is a close relative of arguments of Kahn, Kalai and Linial~\cite{KahKalLin88}, Bourgain, Kahn, Kalai, Katznelson and Linial~\cite{BouKahKalKatLin92}, Friedgut and Kalai~\cite{FriKal96}, and Friedgut~\cite{Fri98,Fri04}, used for the analysis of Boolean functions on product spaces under various conditions.  That connection will be discussed further during the course of the proof.  Perhaps the closest predecessor of Theorems~\ref{thm:bigmain} and~\ref{thm:smallmain} is the theorem of Friedgut that a Boolean function on $\{0,1\}^N$ with controlled total influences is close, in the uniform measure, to a Boolean function depending on a controlled number of coordinates: see~\cite{Fri98}.

Theorems~\ref{thm:bigmain} and~\ref{thm:smallmain} also have relatives in other studies of concentration for product measures and $\ell_\infty$-metrics.  Several earlier works have sought conditions on a metric probability space $(X,\mu)$ under which the isoperimetric function of $(X^N,\mu^N,d^N_\infty)$ may be either estimated or determined exactly, and some quite general results are now known: see, for instance, Barthe~\cite{Barthe04} and the reference given there.  However, those results generally concern the strictly extremal behaviour of the isoperimetric problem in these spaces, for situations in which the minimizing sets can be described exactly, or in which that problem can be shown to behave well under tensorizing.  By contrast, the isoperimetric result descending from Theorem~\ref{thm:bigmain} is very crude:

\begin{thm*}
For any connected and locally connected $(X,d,\mu)$, and any $\delta,\eps > 0$, there is some $q$, depending on the space and on $\delta$ and $\eps$, such that the following holds.  If $A,B \subseteq X^N$ are such that
\[\inf\{d^N_\infty(x,y)\,|\ x\in A,\ y \in B\} \geq \delta,\]
then there are a set $S \subseteq [N]$ with $|S| \leq q$, and subsets $A',B' \subseteq X^N$ that depend only on coordinates in $S$, such that $\mu^N(A\setminus A'),\mu^N(B\setminus B') < \eps$ and 
\[\inf\{d^N_\infty(x,y)\,|\ x\in A',\ y \in B'\} \geq \delta - \eps.\]
\end{thm*}

Thus, in a sense, the approximate isoperimetric problem for $\ell_\infty$-product spaces may be confined to low-dimensional products.

\begin{proof}[Sketch proof]
The $1$-Lipschitz function $f(x):= \min\{\delta,\rm{dist}(x,A)\}$ has the property that $f|A = 0$ and $f|B = \delta$.  Applying Theorem~\ref{thm:bigmain} to $f$ gives a set $S$ of controlled size such that $f \approx E_S f$, and now one can read off $A'$ and $B'$ as suitable level sets of $E_S f$.
\end{proof}

This argument is likely to give a very poor quantitative dependence, and the further details are routine, so they will be omitted.  It is also similar to part of the argument of Dinur, Friedgut and Regev in~\cite{DinFriReg08} concerning the structure of independent sets in graph powers.

\subsubsection*{Acknowledgement}

An earlier version of this paper was significantly strengthened following suggestions by the anonymous referee.

\section{Proof of the special case}

For a differentiable function $f:\bbT^N \to \bbR$ and $n \leq N$, let $\partial_n f$ be the partial derivative in the $n^{\rm{th}}$ coordinate, and let
\[\vpartial f:= (\partial_n f)_{n=1}^N :\bbT^N\to \bbR^N.\]
At any point $x \in \bbT^N$ one may identify the tangent space with $\bbR^N$ in the obvious way.  As one zooms in on that point the metric $d^N_\infty$ converges to the norm of $\ell_\infty^N$.  From this it follows easily that
\begin{eqnarray}\label{eq:Lipnorm}
\rm{Lip}(f) = \sup_{x \in \bbT^N}\|\vpartial f(x)\|_{\ell_1^N} = \sup_{x \in \bbT^N}\Big(\sum_{n=1}^N|\partial_n f(x)|\Big)
\end{eqnarray}
for all $f \in \cal{C}^1(\bbT^N)$.

The method below actually gives a slightly stronger result than Theorem~\ref{thm:smallmain}.  For $p \in [1,\infty)$ and a suitably integrable function $F = (F_1,\ldots,F_N):\bbT^N\to \bbR^N$, define the norm
\[\|F\|_{L_p\ell_p^N} := \Big(\int_{\bbT^N} \|F(x)\|^p_{\ell_p^N}\,\d x\Big)^{1/p}  = \Big(\sum_{n=1}^N\int_{\bbT^N} |F_n(x)|^p\,\d x\Big)^{1/p},\]
where the integrals are with respect to the Haar probability measure on $\bbT^N$.  This norm will be used with $p$ equal to $1$ or $2$.

From each norm $\|\cdot\|_{L_p\ell_p^N}$, one may define a seminorm on $\cal{C}^1(\bbT^N)$ by 
\[f\mapsto \|\vpartial f\|_{L_p\ell_p^N}.\]
When $p = 2$, this is the homogeneous Sobolev seminorm $\|f\|_{\stackrel{\circ}{W}^{1,2}(\bbT^N)}$.

Since $\|\cdot\|_{\ell_p^N} \leq \|\cdot\|_{\ell_1^N}$ for all $p$, any $f \in \cal{C}^1(\bbT^N)$ satisfies
\[\|\vpartial f\|_{L_p\ell_p^N} \leq \rm{Lip}(f) \quad \forall p \in [1,\infty).\]
Using this and the fact that any $1$-Lipschitz function on $\bbT^N$ may be uniformly approximated by smooth $1$-Lipschitz functions, the following immediately implies Theorem~\ref{thm:smallmain}.

\begin{thm}\label{thm:medmain}
For every $\eps > 0$ there is a $q \geq 1$ such that if $f \in \cal{C}^1(\bbT^N)$ satisfies $\|\vpartial f\|_{L_1\ell_1^N} \leq 1$ and $\|f - \int f\|_2 \leq 1$, then there is some $S\subseteq [N]$ with $|S|\leq q$ such that
\[\|f - \sfE_Sf\|_1 < \eps.\]
The same result holds for functions on $[0,1]^N$, possibly with a different dependence of $q$ on $\eps$.
\end{thm}

Most of the work below will go into the first part of Theorem~\ref{thm:medmain}, and that is also the case that will imply Theorem~\ref{thm:bigmain}.  The result for $[0,1]^N$ is an easy corollary, but is included here because it is needed for Theorem~\ref{thm:biLiprig}.

Now let $\Delta:= \sum_{n\leq N}\partial_n^2$, the non-positive definite Laplacian.  Let $P_t := \exp (t \Delta)$, $t \geq 0$, be the resulting heat semigroup, which is well-defined and self-adjoint on $L_2(\bbT^N)$.  As is standard, this is the semigroup corresponding to Brownian motion on $\bbT^N$. For each $t > 0$ it may be written explicitly in terms of the Gaussian measure  $\g_t$ on $\bbR$ with variance $t$:
\begin{eqnarray}\label{eq:rep-P}
P_tf(x) = \int_{\bbR^N}f(x+\ol{y})\,\g^{\otimes N}_t(\d y),
\end{eqnarray}
where $\ol{y}$ denotes the image of $y$ under the obvious quotient map $\bbR^N\to \bbT^N$.

Every $P_t$ is a contraction for the norm $\|\cdot\|_{L_p\ell_p^N}$ for any $p \in [1,\infty]$.  A simple calculation gives $\partial_n P_t f = P_t \partial_n f$ for all $t$, $n$ and differentiable $f$, and hence
\begin{eqnarray}\label{eq:Lipineq}
\|\vpartial P_t f\|_{L_p\ell_p^N} \leq \|\vpartial f\|_{L_p\ell_p^N}
\end{eqnarray}
for any $p$, $f$ and $t \geq 0$.

The proof of Theorem~\ref{thm:medmain} follows similar lines to the proofs of the main results in~\cite{KahKalLin88,BouKahKalKatLin92,FriKal96,Fri98,Fri04}, concerning the structure of Boolean functions with low total influences on $\{0,1\}^N$ and $[0,1]^N$.  Note, however, that the obvious analog of Theorem~\ref{thm:smallmain} for Boolean functions --- that Boolean functions with small total influences are close to juntas --- is false for functions on $[0,1]^N$: see Hatami~\cite{Hatami09,Hatami12} for the precise structure in that case.
\begin{itemize}
\item On the one hand, if $\|\vpartial f\|_{L_1\ell_1^N} \leq 1$, then the evolution $t\mapsto P_t f$ is not too fast in the norm $\|\cdot\|_1$, so that for small $t$ one has $P_tf \approx f$.

\item On the other hand, a hypercontractivity inequality for the semigroup $(P_t)_{t\geq 0}$ implies that if $\|\partial_n f\|_1 \ll 1$ for some $n$, then $\|\partial_n P_t f\|_2$ rapidly decays towards zero.  This will occur so fast that even for small $t$, the derivatives $\partial_n P_t f$ are extremely small in $\|\cdot\|_2$ for all but a few choices of $n$, so that $P_t f$ is close to a function that depends on only those few coordinates.
\end{itemize}
Thus, one may first approximate $f$ by $P_t f$ for some small $t > 0$, and then prove that $P_tf$ is close to a function of only a few coordinates.

The first of these two steps results from the following.

\begin{lem}\label{lem:Lipbound}
If $f \in \cal{C}^1(\bbT^N)$ then
\[\|f - P_tf\|_1 \leq \sqrt{t}\|\vpartial f\|_{L_1\ell_1^N} \quad \forall t \geq 0.\]
\end{lem}

\begin{proof}
This is easily deduced from the representation~(\ref{eq:rep-P}), which gives
\[\|f - P_tf\|_1 \leq \int_{\bbT^N} \int_{\bbR^N}|f(x) - f(x + \ol{y})|\,\g_t^{\otimes N}(\d y)\,\d x.\]
By the Fundamental Theorem of Calculus and a change of variables, this last bound equals
\begin{multline*}
\int_{\bbT^N} \int_{\bbR^N}\Big|\int_0^1 \frac{\d \phantom{s}}{\d s}f(x + \ol{sy})\,\d s\Big|\,\g_t^{\otimes N}(\d y)\,\d x\\ \leq \int_{\bbT^N} \int_0^1 \int_{\bbR^N}\Big|\sum_{n=1}^N y_n\partial_nf(x + \ol{sy})\Big|\,\g_t^{\otimes N}(\d y)\,\d s\,\d x\\
= \int_{\bbT^N} \int_{\bbR^N}\Big|\sum_{n=1}^N y_n\partial_nf(x')\Big|\,\g_t^{\otimes N}(\d y)\,\d x'.
\end{multline*}
By the monotonicity of Lebesgue norms, this is at most
\[\int_{\bbT^N} \Big(\int_{\bbR^N}\Big(\sum_{n=1}^N y_n\partial_nf(x')\Big)^2\,\g_t^{\otimes N}(\d y)\Big)^{1/2}\,\d x',\]
and now, since the coordinates $y_n$ of $y$ are independent under $\g_t^{\otimes N}$, this bound equals
\[\Big(\int_{\bbR^N}y_1^2\,\g_t(\d y_1)\Big)^{1/2}\int_{\bbT^N} \Big(\sum_{n=1}^N |\partial_nf(x')|^2 \Big)^{1/2}\,\d x' = \sqrt{t}\int_{\bbT^N} \|\vec{\partial}f(x')\|_{\ell_2^N}\,\d x'.\]
Finally, this is at most $\sqrt{t}\|f\|_{L_1\ell_2^N}$, since $\|\,\cdot\,\|_{\ell_2^N} \leq \|\,\cdot\,\|_{\ell_1^N}$.
\end{proof}

For the second part of the proof, the key ingredients are a reverse Poincar\'e inequality and a hypercontractive estimate for the heat semigroup on a torus.

\begin{prop}[Reverse Poincar\'e inequality]\label{prop:reverse}
If $f \in \cal{C}^1(\bbT^N)$ then
\[\|\vec{\partial}P_t f\|_{L_2\ell_2^N} \leq \|f\|_2/\sqrt{t} \quad \forall t > 0.\]
\end{prop}

\begin{proof}
For any $f \in \cal{C}^1(\bbT^N)$, the self-adjointness of $P_t$ and integration by parts give
\begin{multline*}
\|\vec{\partial}P_t f\|^2_{L_2\ell_2^N} = \sum_{n=1}^N \int_{\bbT^N} (\partial_n P_tf(x))^2\,\d x = -\sum_{n=1}^N \int_{\bbT^N} P_tf(x)\cdot \partial^2_n P_tf(x)\,\d x\\ = -\sum_{n=1}^N \int_{\bbT^N} \rme^{t\Delta}f(x)\cdot \partial^2_n \rme^{t\Delta}f(x)\,\d x =\int_{\bbT^n} f(x)\cdot (-\Delta)\rme^{2t\Delta}f(x)\,\d x.
\end{multline*}
Since $-\Delta$ has spectrum contained in $[0,\infty)$, the Spectral Theorem gives
\[\|(-\Delta)\rme^{2t\Delta}f\|_2 \leq (\max_{x \in [0,\infty)}x\rme^{-2tx})\|f\|_2 \leq \|f\|_2/t.\]
Substituting above, the Cauchy--Bunyakowski--Schwartz inequality gives
\[\|\vec{\partial}P_t f\|^2_{L_2\ell_2^N}  \leq \|f\|_2\|(-\Delta)\rme^{2t\Delta}f\|_2\leq \|f\|_2^2/t.\]
\end{proof}

\begin{prop}[Hypercontractivity on tori]\label{prop:hypercon}
Let $t > 0$ and $p := 1 + \rm{e}^{-2t}\in (1,2)$. If $f \in L_p(\bbT^N)$, then
\[\|P_t f\|_2 \leq \|f\|_p.\]
\end{prop}

\begin{proof}
In case $N=1$, this is a special case of Weissler's hypercontractive estimates in~\cite[Theorem 2]{Weissler80}.  Given this, the result for general $N$ follows because the operator $P_t$ is a tensor product of one-dimensional operators, and all the norms $\|\cdot\|_p$ also tensorize.
\end{proof}

Weissler's proof of the case $N=1$ closely follows Gross' famous work~\cite{Gross75} on the Ornstein-Uhlenbeck semigroup and Nelson's hypercontractive estimates.  The main task is to prove a logarithmic Sobolev inequality for the Laplacian on the circle; this can then be integrated over a time interval to give hypercontractivity.  Weissler's proof of the logarithmic Sobolev inequality uses Fourier analysis, but a more direct argument from standard properties of the heat semigroup is also possible: c.f.~\cite[Theorem 5.1]{Led--book}.

Lastly, the proof will also need the following simple Poincar\'e inequality.

\begin{lem}\label{lem:vceest}
For any $S \subseteq [N]$ and $f \in \cal{C}^1(\bbT^N)$ one has
\[\|f - \sfE_Sf\|^2_2 \leq \sum_{n \in [N]\setminus S}\|\partial_n f\|_2^2.\]
\end{lem}

\begin{proof}
If $S = [N]\setminus \{n\}$, then the desired inequality is
\[\|f - \sfE_Sf\|_2^2 \leq \|\partial_n f\|_2^2.\]
This follows by applying the Poincar\'e inequality on $\bbT$ to each of the one-dimensional slices
\[f(x_1,\ldots,x_{n-1},\,\cdot\,,x_{n+1},\ldots,x_N), \quad (x_1,\ldots,x_{n-1},x_{n+1},\ldots,x_N) \in \bbT^{N-1},\]
and then integrating over $\bbT^{N-1}$.

For the general case, observe that for any $n$ and $S$, the operators $\partial_n$ and $\sfE_S$ commute, and so one has
\[\|\partial_n \sfE_Sf\|_2 = \|\sfE_S(\partial_nf)\|_2 \leq \|\partial_n f\|_2.\]
Now, if $S = [N]\setminus \{n_1,\ldots,n_M\}$ and one defines $f_0 := f$ and then $f_k := \sfE_{[N]\setminus \{n_k\}}f_{k-1}$ for $k=1,\ldots,M$, then this sequence is a reverse martingale, so we have
\[\|f - \sfE_Sf\|^2_2 = \sum_{k=1}^M\|f_{k-1} - f_k\|_2^2\leq \sum_{k=1}^M\|\partial_{n_k} f_{k-1}\|_2^2 \leq  \sum_{n \in [N]\setminus S}\|\partial_n f\|_2^2.\]
\end{proof}

\begin{lem}\label{lem:approxatpost}
Suppose that $f \in \cal{C}^1(\bbT^N)$ has $\|\vpartial f\|_{L_1\ell_1^N} \leq 1$ and $\|f\|_2 \leq 1$.  Fix $t > 0$ and $\eta > 0$, and let
\[S := \{n \in [N]\,|\ \|\partial_nf\|_1 \geq \eta\}.\]
Then
\[\|P_{2t}f - \sfE_SP_{2t}f\|_2 < t^{-\frac{\rm{e}^{-2t}}{1 + \rm{e}^{-2t}}} \cdot \eta^{\frac{1 - \rm{e}^{-2t}}{2(1 + \rm{e}^{-2t})}}.\]
\end{lem}

\begin{proof}
Proposition~\ref{prop:hypercon} and the log-convexity of the Lebesgue norms give
\[\|P_{2t}(\partial_n f)\|_2 = \|P_t(\partial_n P_tf)\|_2\leq \|\partial_n P_t f\|_{1 + \rm{e}^{-2t}} \leq \|\partial_n P_tf\|_2^{\frac{2\rm{e}^{-2t}}{1 + \rm{e}^{-2t}}}\|\partial_n P_tf\|_1^{\frac{1 - \rm{e}^{-2t}}{1 + \rm{e}^{-2t}}}.\]
Substituting this into Lemma~\ref{lem:vceest} gives
\begin{multline*}
\|P_{2t} f - \sfE_SP_{2t} f\|_2^2 \leq \sum_{n \in [N]\setminus S}\|P_{2t}(\partial_n f)\|_2^2 \leq \sum_{n \in [N]\setminus S}\|\partial_n P_tf\|_2^{\frac{4\rm{e}^{-2t}}{1 + \rm{e}^{-2t}}}\|\partial_n P_tf\|_1^{\frac{2(1 - \rm{e}^{-2t})}{1 + \rm{e}^{-2t}}}\\
\leq \Big(\sum_{n \in [N]\setminus S}\|\partial_n P_tf\|_2^2\Big)^{\frac{2\rm{e}^{-2t}}{1 + \rm{e}^{-2t}}}\Big(\sum_{n \in [N]\setminus S}\|\partial_n P_tf\|_1^2\Big)^{\frac{1 - \rm{e}^{-2t}}{1 + \rm{e}^{-2t}}},
\end{multline*}
where the last bound follows from H\"{o}lder's Inequality with exponents $(\frac{1 + \rm{e}^{-2t}}{2\rm{e}^{-2t}},\frac{1 + \rm{e}^{-2t}}{1 - \rm{e}^{-2t}})$.  Applying the contractivity of $P_t$ and then the definition of $S$ to the second factor, we may now bound this by
\begin{multline*}
\Big(\sum_{n \in [N]\setminus S}\|\partial_n P_tf\|_2^2\Big)^{\frac{2\rm{e}^{-2t}}{1 + \rm{e}^{-2t}}}\Big(\sum_{n \in [N]\setminus S}\|\partial_n f\|_1^2\Big)^{\frac{1 - \rm{e}^{-2t}}{1 + \rm{e}^{-2t}}}\\
\leq \Big(\sum_{n \in [N]\setminus S}\|\partial_n P_tf\|_2^2\Big)^{\frac{2\rm{e}^{-2t}}{1 + \rm{e}^{-2t}}}\Big(\eta\sum_{n \in [N]\setminus S}\|\partial_n f\|_1\Big)^{\frac{1 - \rm{e}^{-2t}}{1 + \rm{e}^{-2t}}}\\
\leq \eta^{\frac{1 - \rm{e}^{-2t}}{1 + \rm{e}^{-2t}}}(\|\vec{\partial}P_tf\|_{L_2\ell_2^N}^2)^{\frac{2\rm{e}^{-2t}}{1 + \rm{e}^{-2t}}}(\|\vec{\partial} f\|_{L_1\ell_1^N})^{\frac{1 - \rm{e}^{-2t}}{1 + \rm{e}^{-2t}}}.
\end{multline*}
Finally, our assumptions on $f$ and Proposition~\ref{prop:reverse} bound this by
\[\eta^{\frac{1 - \rm{e}^{-2t}}{1 + \rm{e}^{-2t}}}(t^{-1}\|f\|^2_2)^{\frac{2\rm{e}^{-2t}}{1 + \rm{e}^{-2t}}}(\|\vec{\partial} f\|_{L_1\ell_1^N})^{\frac{1 - \rm{e}^{-2t}}{1 + \rm{e}^{-2t}}} \leq t^{-\frac{2\rm{e}^{-2t}}{1 + \rm{e}^{-2t}}}\cdot \eta^{\frac{1 - \rm{e}^{-2t}}{1 + \rm{e}^{-2t}}}.\]
Taking square roots completes the proof.
\end{proof}

\begin{proof}[Proof of Theorem~\ref{thm:medmain}]
Consider $f \in \cal{C}^1(\bbT^N)$.  Replacing $f$ with $f - \int f$ if necessary, we may assume that $\int f = 0$.  Let $t,\eta > 0$, and let $S$ be as in Lemma~\ref{lem:approxatpost}.  Combining Lemmas~\ref{lem:Lipbound} and~\ref{lem:approxatpost} gives
\begin{multline*}
\|f - \sfE_Sf\|_1 \leq \|f - P_{2t} f\|_1 + \|P_{2t} f - \sfE_SP_{2t}f\|_1 + \|\sfE_S(f - P_{2t} f)\|_1 \\ \leq 2\|f - P_{2t}f\|_1 + \|P_{2t} f - \sfE_SP_{2t}f\|_2 \leq 2\sqrt{2t} + t^{-\frac{\rm{e}^{-2t}}{1 + \rm{e}^{-2t}}} \cdot \eta^{\frac{1 - \rm{e}^{-2t}}{2(1 + \rm{e}^{-2t})}}.
\end{multline*}

For any $\eps > 0$, choose $t \in (0,\eps^2/32)$, so that
\[2\sqrt{2t} < \eps/2;\]
then choose $\eta > 0$ so small that
\[t^{-\frac{\rm{e}^{-2t}}{1 + \rm{e}^{-2t}}} \cdot \eta^{\frac{1 - \rm{e}^{-2t}}{2(1 + \rm{e}^{-2t})}} < \eps/2.\]
For this choice of $t$ and $\eta$, one obtains
\[\|f - \sfE_Sf\|_1 < \eps.\]
On the other hand, the definition of $S$ gives
\[\eta|S| \leq \sum_{n \in S}\|\partial_n f\|_1 = \|\vpartial f\|_{L_1 \ell_1^N} \leq 1,\]
so $|S| \leq 1/\eta$, which is bounded only in terms of $\eps$.

Lastly, suppose instead that $f \in \cal{C}^1([0,1]^N)$.  Let $F:\bbT\to [0,1]$ be the map defined by
\[F(\theta + \bbZ) = \left\{\begin{array}{ll}2\theta & \quad \hbox{if}\ \theta \in [0,1/2)\\ 2 - 2\theta & \quad \hbox{if}\ \theta \in [1/2,1).\end{array}\right.\]
Then $F^{\times N}:\bbT^N\to [0,1]^N$ is $2$-Lipschitz, differentiable almost everywhere and measure preserving, so $f\circ F^{\times N}$ is differentiable almost everywhere and
\[\|\vpartial (f\circ F^{\times N})\|_{L_p\ell_p^N}\leq 2\|\vpartial f\|_{L_p\ell_p^N} \quad \forall p \in [1,\infty).\]

Therefore $f\circ F^{\times N}$ may be uniformly approximated by continuously differentiable functions satisfying the same inequalities, for example by convolving with a mollifier. Since also $(E_Sf)\circ F^{\times N} = E_S(f\circ F^{\times N})$, the proof is completed by applying the first part of the result to $f\circ F^{\times N}$.
\end{proof}

Theorem~\ref{thm:medmain} begs the following question.  I suspect the answer is negative, but have not been able to find a counterexample.

\begin{ques}\label{ques}
In the statement of Theorem~\ref{thm:medmain}, is it enough to assume that $\|\vpartial f\|_{L_1\ell_1^N} \leq 1$, without the bound on $\|f - \int f\|_2$?  (Of course one should expect a worse dependence of $q$ on $\eps$.)
\end{ques}

\begin{rmks}
\emph{1.}\quad The argument above can also be used, essentially without change, to prove an analog of Theorem~\ref{thm:medmain} for the standard Gaussian measure $\g$ on $\bbR^N$: a function $f \in \cal{C}^1(\bbR^N)$ for which $\int \|\vec{\partial} f\|_{\ell_1^N}\,\d\g \leq 1$ and $\|f\|_{L_2(\g)} \leq 1$ may be approximated in $L_1(\g)$ by functions of boundedly many coordinates.

Indeed, it seems likely that these methods, and Theorem~\ref{thm:medmain}, generalize to a large class of Markov diffusion semigroups subject to suitable `curvature' conditions, as studied in~\cite{Led00,BakGenLed14}.  More general reverse Poincar\'e inequalities, for instance, can be found in~\cite[Section 4.7]{BakGenLed14}, and hypercontractivity estimates in~\cite[Section 5.2]{BakGenLed14}.

\vspace{7pt}

\emph{2.}\quad Another connection worth remarking is with the recent works~\cite{KelMosSen12} and~\cite{CorLed12}.  These establish versions of Talagrand's famous variance-bound for functions on $\{\pm 1\}^N$ (see~\cite{Tal94}) in various new settings, including some product and non-product measures on $\bbR^N$, using a similar strategy to that above.  It should also be possible to deduce Theorem~\ref{thm:medmain} directly from one of those Talagrand-type inequalities, such as~\cite[Theorem 1]{CorLed12}. \fin
\end{rmks}

\section{Proof of the general case}

\begin{lem}\label{lem:approxbyLip}
Let $(X,d)$ be a compact metric space and let $f:X\to\bbR$ be continuous with a non-decreasing modulus of continuity $\omega$.  Let $\eps > 0$, and let $K_1 := \sup_{r \geq \eps}r^{-1}\omega(r) \in [0,\infty]$ and $K= \max\{K_1,1\}$. Then there is a $K$-Lipschitz function $h:X\to\bbR$ such that $\|f - h\|_\infty \leq K\eps$.
\end{lem}

\begin{proof}
This is vacuous if $K = \infty$, so assume $K$ is finite.  For each $n \in \bbZ$ let
\[X_n := \{x \in X\,|\ f(x) \leq \eps n\},\]
and define $h:X\to \bbR$ by
\[h(x) := \inf_{n\in \bbZ}(\eps n + Kd(x,X_n)).\]
As a pointwise infimum of $K$-Lipschitz functions, $h$ is $K$-Lipschitz.  If $f(x) \leq \eps n$, then $x \in X_n$, and so $h(x) \leq \eps n + Kd(x,X_n) = \eps n$.  Infimizing over $n$, this gives $h(x) \leq f(x) + \eps$.

On the other hand, for any $x \in X$, $n \in \bbZ$ and $y \in X_n$, one has
\[f(x) \leq f(y) + \omega(d(x,y)) \leq f(y) + \omega(d(x,y) + \eps) \leq f(y) + Kd(x,y) + K\eps.\]
Using that $f(y) \leq \eps n$ and then infimizing over $y \in X_n$, this gives
\[f(x) \leq \eps n + Kd(x,X_n) + K\eps.\]
Now infimizing over $n$ gives $f(x) \leq h(x) + K\eps$, as required.
\end{proof}

\begin{proof}[Proof of Theorem~\ref{thm:bigmain}]
As Gromov reminds us in the proof of the Non-Dissipation Theorem in~\cite[Section 3$\frac{1}{2}.62$]{Gro01}, since $(X,d,\mu)$ is connected and locally connected, there is a uniformly continuous map $F:\bbT\to X$ such that $F_\ast m = \mu$.  Let $\s$ be a non-decreasing modulus of continuity for $F$.  For any $(x_n)_n,(y_n)_n \in X^N$, one has
\[\max_{n\leq N}d(F(x_n),F(y_n)) \leq \max_{n \leq N}\s(|x_n - y_n|) = \s (|(x_n)_n - (y_n)_n|_\infty^N),\]
so $F^{\times N}:\bbT^N\to X^N$ has the same modulus of continuity $\s$ for all $N$ for the metrics $|\cdot|_\infty^N$ and $d_\infty^N$.

Now let $f:X^N\to \bbR$ have non-decreasing modulus of continuity $\omega$ for the metric $d_\infty^N$.  By the above, $f\circ F^{\times N}:\bbT^N\to \bbR$ has non-decreasing modulus of continuity $\omega\circ \s$.  Let $\eps > 0$, and let $K_\eps := \max\{\sup_{r \geq \eps}r^{-1}\omega(\s(r)),1\}$, as in Lemma~\ref{lem:approxbyLip} for the function $f$.  That lemma gives a $K_\eps$-Lipschitz function $h:\bbT^N\to\bbR$ for which
$\|f\circ F^{\times N} - h\|_\infty \leq K_\eps\eps$.

Having found this $h$, Theorem~\ref{thm:smallmain} gives a set $S \subseteq [N]$ with $|S|$ bounded in terms of $\eps$ and $K_\eps$ (hence, in terms of $\eps$, $\omega$ and $\s$) such that
\[\|h - \sfE_Sh\|_1 \leq \eps,\]
and combining this with the previous inequalities gives
\[\|f - \sfE_Sf\|_1 \leq 2\|f\circ F^{\times N} - h\|_1 + \|h - \sfE_Sh\|_1 \leq 2K_\eps\eps + \eps.\]
An easy exercise shows that $K_\eps\eps \to 0$ as $\eps \to 0$, so this completes the proof.
\end{proof}

\section{Lipschitz maps between Hamming cubes}

\begin{proof}[Proof of Theorem~\ref{thm:biLiprig}]
Suppose that $F = (F_1,\ldots,F_M):[0,1]^N\to [0,1]^M$ is $L$-Lipschitz between the Hamming metrics.  It may be uniformly approximated by smooth $L$-Lipschitz functions, so assume smoothness also.  The $L$-Lipschitz bound implies that
\[\sum_{m=1}^M |F_m(x + \d x) - F_m(x)| \leq L\sum_{n=1}^N |\d x_n|\]
for all $x \in [0,1]^N$ and all perturbations $\d x = (\d x_1,\ldots,\d x_N)$.

For fixed $n \leq N$, this inequality may be applied with $\d x = (0,\ldots,0,\d x_n,0,\ldots,0)$ to obtain
\[\max_{m\leq M}|F_m(x + \d x) - F_m(x)| \leq \sum_{m = 1}^M |F_m(x + \d x) - F_m(x)| \leq L|\d x_n|.\]
Normalizing and letting $|\d x_n| \to 0$, this implies firstly that $\|\partial_n F_m\|_{L_\infty} \leq L$ for all $n$ and $m$, and secondly that
\[\sum_{m=1}^M \sum_{n=1}^N|\partial_n F_m(x)| \leq LN = L M/\a \quad \Longrightarrow \quad \frac{1}{M}\sum_{m=1}^M \|\vpartial F_m\|_{L_1\ell_1^N} \leq L/\a.\]

Given $\eps > 0$, let
\[I := \big\{m \leq M\,\big|\ \|\vpartial F_m\|_{L_1\ell_1^N} \leq 2L/\a\eps\big\}.\]
Applying Chebyshev's Inequality, the last inequality above implies that $|I| \geq (1-\eps/2)M$.

For each $m \in I$, one has $\|\vpartial F_m\|_{L_1\ell_1^N} \leq 2L/a\eps$, and of course $F_m$ is uniformly bounded by $1$. Therefore, for each $m \in I$, the second part of Theorem~\ref{thm:medmain} gives a subset $S_m \subseteq [N]$ of size bounded in terms of $L$, $\eps$ and $\a$ such that
\[\|F_m - E_{S_m}F_m\|_1 \leq \eps/2.\]

Let $G_m := E_{S_m}F_m$ for $m \in I$, and let $G_m$ be any constant function in case $m \in [M]\setminus I$.  Then the above inequalities combine to give
\[\int_{[0,1]^N}|F(x) - G(x)|^M_1\,\d x \leq \frac{1}{M}\sum_{m\in I}\|F_m - E_{S_m}F_m\|_1 + \frac{M - |I|}{M} < \eps/2 + \eps/2 = \eps.\]
\end{proof}

Equivalently, Theorem~\ref{thm:biLiprig} asserts that most of the functions $F_m$ are individually close to functions that depend on only small sets of coordinates.  One cannot tighten this conclusion to apply to strictly every $F_m$.  For example, the function
\[[0,1]^N \to [0,1]^N:(x_1,\ldots,x_N) \mapsto \big(x_1,\ldots,x_{N-1},\sin(2\pi(x_1 + \cdots + x_N))\big)\]
is easily checked to be continuously differentiable and $2$-Lipschitz for the Hamming metrics, but its last coordinate is not close to any function depending on fewer than $N$ coordinates.  By concatenating several such examples, for any divergent positive sequence $a_N = \rm{o}(N)$, one may construct a sequence of $2$-Lipschitz maps in which roughly $a_N$ of the output coordinates each depend on roughly $N/a_N$ of the input coordinates.

A similar argument to the proof of Theorem~\ref{thm:biLiprig}, using Friedgut's Theorem~\cite{Fri98} on Boolean functions with small influences in place of Theorem~\ref{thm:medmain}, shows that the analog of Theorem~\ref{thm:biLiprig} also holds for functions
\[\{0,1\}^N \to \{0,1\}^M,\]
where domain and target are again given the Hamming metrics.  This argument appeared recently in Subsection 1.1 of~\cite{BenCohShi--HammbiLip}, where it was used to analyze a particular bi-Lipschitz bijection between $\{0,1\}^N$ and a radius-$(N/2)$ Hamming ball in $\{0,1\}^{N+1}$ (although those authors were also able to give a more precise bespoke analysis for that function).  Actually, the argument using Friedgut's Theorem in~\cite{BenCohShi--HammbiLip} requires only `bounded average stretch' for the function $\{0,1\}^N \to \{0,1\}^M$.  In our setting of a map $F = (F_1,\ldots,F_m):[0,1]^N\to [0,1]^M$, this corresponds to assuming that
\[\sum_{m=1}^M \|\vec{\partial}F_m\|_{L_1\ell_1^N} \leq LM\]
for some fixed constant $L$.  One can therefore obtain a closer analog of that conclusion from~\cite{BenCohShi--HammbiLip} by using the full strength of Theorem~\ref{thm:medmain}.

By contrast, no obvious analog of Theorem~\ref{thm:biLiprig} holds between Euclidean spheres with the Euclidean metrics.  To see this, consider the following two natural classes of Lipschitz map $\rm{S}^{3N-1}\to \rm{S}^{3N-1}$:
\begin{itemize}
\item maps of the form
\[(x_1,\ldots,x_{3N}) \mapsto \big(F_1(x_1,x_2,x_3),\ldots,F_N(x_{3N-2},x_{3N-1},x_{3N})\big),\]
where $F_i:\bbR^3\to \bbR^3$ for $i\leq N$ are $10$-Lipschitz functions such that $\|F_i(u,v,w)\|_{\ell_2^3} = \|F_i(u,v,w)\|_{\ell_2^3}$ (for example, one may choose a sequence of $10$-Lipschitz maps $\rm{S}^2\to \rm{S}^2$ and extend them radially to $\bbR^3$);
\item orthogonal rotations.
\end{itemize}
Both of these classes are individually somewhat simple.  However, one may now form a composition of a few maps drawn alternately from these two classes, and easily produce a map with the property that no output coordinate is close to any function that depends on only a low-dimensional projection of the domain sphere.  It is also easy to see that such compositions can be chosen to be $\rm{O}(1)$-bi-Lipschitz and preserve Lebesgue measure on $\rm{S}^{3N-1}$, by a suitable choice of the functions $F_i$ appearing in the examples of the first kind.

\section{Some questions about non-product measures}

Several important developments in the theory of concentration have come from extensions of concentration inequalities from product to non-product measures.  Similarly, it would be interesting to know whether any other natural measures on product spaces enjoy an analog of Theorem~\ref{thm:bigmain}.

One rich source of examples is dynamics. Suppose that $(X,d,\mu)$ is a connected and locally connected metric probability space and that $T:X\to X$ is a $\mu$-preserving homeomorphism.  Then for each $N$ one may form the topological embeddings
\[T^{[0;N-1]}:= (\rm{id},T,T^2,\ldots,T^{N-1}):X\to X^N,\]
and consider the image $T^{[0;N-1]}(X)$ and pushforward measure $T^{[0;N-1]}_\ast \mu$.  Equivalently, one may consider the original measure space $(X,\mu)$ with the pullback of the metric $d^N_\infty$, giving the new metric
\[D^N_\infty(x,y) := \max_{0 \leq n \leq N-1}d(T^nx,T^ny).\]
This construction is important in the theory of dynamical systems: the exponential growth rate of the $D^N_\infty$-covering numbers of $X$ gives the Bowen-Dinaburg definition of topological entropy.

One may obtain examples similar to $\ell_\infty$-product spaces with product measures in this way.  For instance, if $(X,\mu,T)$ is the shift on $\bbT^\bbZ$ with its Haar measure and some sensible choice of compact metric, then the above metrics $D^N_\infty$ behave increasingly like the metrics $|\cdot|_\infty^N$ on finite marginals of $\bbT^\bbZ$.  This suggests the following question.

\begin{ques}
For which systems $(X,\mu,T)$ is it the case that $\forall \eps > 0$ $\exists p,K \geq 1$ such that $\forall N \geq 1$, if $f:X\to \bbR$ is $1$-Lipschitz for the metric $D^N_\infty$, then it is $\eps$-close in $\|\cdot\|_{L_2(\mu)}$ to a function of the form
\[F(T^{n_1}x,T^{n_2}x,\ldots,T^{n_p}x)\]
for some $n_1,\ldots,n_p \in \{0,1,\ldots,N-1\}$, where $F:X^p\to \bbR$ is $K$-Lipschitz for the metric $d^p_\infty$?
\end{ques}

Another instructive example is the map $\theta \mapsto 2\theta$ on $(\bbT,|\cdot|,m_\bbT)$.  An easy exercise shows that for this map, $D^N_\infty$ is equivalent up to constants to the metric
\[\rho^N(\theta,\theta') := \min\{2^N|\theta - \theta'|,1\}.\]
Therefore any function $\bbT^N \to [0,1]$ which is $2^N$-Lipschitz for the usual metric $|\cdot|$ becomes Lipschitz with bounded Lipschitz constant for the metric $D^N_\infty$.  This certainly includes many maps that cannot be approximated in measure by `low-dimensional' functions $F$ as above, so the $(\times 2)$-map is not an example.  Similar reasoning seems to apply to Anosov diffeomorphisms, so these in general do not give examples.

\begin{ques}
Are there any positive examples for the previous question in which $(X,T)$ has finite topological entropy?
\end{ques}

\bibliographystyle{alpha}
\bibliography{bibfile}

\parskip 0pt
\parindent 0pt

\vspace{7pt}

\end{document}